\documentclass[11pt,a4]{article}
\usepackage[colorlinks=true, pdfstartview=FitV, linkcolor=blue, citecolor=blue,
urlcolor=blue]{hyperref}
\usepackage[applemac]{inputenc}
\usepackage{amssymb,amsmath,amsthm,amsfonts,amscd,amstext}
\usepackage{bbm}
\usepackage[mathscr]{eucal}
\usepackage{multirow}
\usepackage{xypic}
\usepackage{array}
\usepackage{graphicx}
\usepackage{url}
\newtheorem{thm}{Theorem}[section]

\newtheorem{prop}[thm]{Proposition}

\newtheorem{rem}{\it Remark\/}
\def\End{\operatorname{End}} 

\def\Hom{\operatorname{Hom}} 
\def\id{\operatorname{id}}  

\def\P{{\mathbb P}}  
\def\pr{\operatorname{pr}} 

\def\Q{{\mathbb Q}}   
\def\R{{\mathbb R}}    

\def\N{{\mathbb N}}    

\begin{document}
\title{\bf Operads in Itô calculus}
\author
{Roland Friedrich\thanks{Saarland University, friedrich@math.uni-sb.de. The author is supported by the  ERC advanced grant ``Noncommutative distributions in free probability", held by Roland Speicher. 
}}
\date{\today}
\maketitle
\begin{abstract}
We discuss algebraic and geometric properties of the Itô calculus.
\end{abstract}
\tableofcontents
\section{Introduction}

Let us consider the integration by parts formul{\ae} for the Itô $\bullet$ and the Stratonovich $\circ$ integral applied to the subspace of continuous semi-martingales starting at the origin. We have
\begin{eqnarray*}
X\cdot Y&=&X\bullet Y+Y\bullet X+[X,Y],\\
X\cdot Y&=& X\circ Y+Y\circ X,
\end{eqnarray*}
which gives a splitting of the associative and commutative product $\cdot$ into three and two operations, respectively.
Further, the Stratonovich integral is expressible as a $\Q$-linear combination of the operations which we obtain in the splitting of the Itô integral:
\begin{equation*}
X\circ Y=X\bullet Y+\frac{1}{2}[X,Y].
\end{equation*}
To start with, let us consider the associativity condition $(X\cdot Y)\cdot Z=X\cdot(Y\cdot Z)$ for the Itô integral:
\begin{eqnarray*}
(X\cdot Y)\cdot Z &=&\underbrace{(X\bullet Y)\bullet Z}_{1}+\underbrace{(Y\bullet X)\bullet Z}_{2}+\underbrace{[X,Y]\bullet Z}_{6}+\\
&& \underbrace{Z\bullet(X\bullet Y)+Z\bullet(Y\bullet X)+Z\bullet[X,Y]}_{3}+\\
&&\underbrace{[X\bullet Y,Z]}_{5}+\underbrace{[Y\bullet X,Z]}_{4}+\underbrace{[[X,Y],Z]}_{7}
\end{eqnarray*}
\begin{eqnarray*}
X\cdot (Y\cdot Z) &=&\underbrace{X\bullet (Y\bullet Z)+X\bullet(Z\bullet Y)+X\bullet [Y,Z]}_{1}+\\
&& \underbrace{(Y\bullet Z) \bullet X}_{2}+\underbrace{(Z\bullet Y)\bullet X}_{3}+\underbrace{[Y,Z]\bullet X}_{4}+\\
&&\underbrace{[X,Y\bullet Z]}_{6}+\underbrace{[X,Z\bullet Y]}_{5}+\underbrace{[X,[Y,Z]]}_{7}
\end{eqnarray*}
By ``equating" the expressions with the corresponding numbers and appropriately replacing the operator symbols, one recognises the defining relations of Loday and Ronco's  {\bf tridendriform algebra}~\cite{LR04}, as listed e.g. in~[\cite{LV}, p. 531]

Similarly, for the Stratonovich integral, we deduce from the associativity condition $(X\cdot Y)\cdot Z=X\cdot(Y\cdot Z)$ the following relations:
\begin{eqnarray*}
(X\cdot Y)\cdot Z&=&\underbrace{(X\circ Y)\circ Z}_{1}+\underbrace{(Y\circ X)\circ Z}_{2}+\underbrace{Z\circ(X\circ Y)+Z\circ(Y\circ X)}_{3}\\
X\cdot(Y\cdot Z)&=&\underbrace{X\circ(Y\circ Z)+X\circ(Z\circ Y)}_{1}+\underbrace{(Y\circ Z)\circ X}_{2}+\underbrace{(Z\circ Y)\circ X}_{3}
\end{eqnarray*}
Again, by ``equating" the expressions $1,2,3$ and replacing the operator symbols, one recognises the relations of Loday's  {\bf dendriform algebra}~\cite{L95}. 

To keep this note concise we do not further discuss related topics such as, e.g. Rota-Baxter operators here, nor do we introduce natural generalisations of the material presented, which we shall do in~\cite{F}.  
\section{Categorical preliminaries}
Topological spaces and measurable spaces have many properties in common. In order to illustrate this, we list the corresponding notions, as compiled in, e.g.~\cite{W}, for the category of topological spaces. For some of the measure theoretical notions, cf.~\cite{E}.

\begin{prop} The class of all measurable spaces forms a category $\mathbf{Meas}$, with objects the measurable spaces and morphisms the measurable maps.
\end{prop}

The forgetful functor $U:\mathbf{Meas}\rightarrow\mathbf{Set}$, 
is given by assigning to every measurable space the underlying set and to every measurable map the corresponding set-theoretic function.

Further, for the forgetful functor $U$ both the left $L$ and right adjoint $R$ exist. So, the {\bf left adjoint}
$$
L:\mathbf{Set}
\rightarrow\mathbf{Meas}
$$
endows a non-empty set with the {\bf powerset $\sigma$-algebra}, i.e. $\mathfrak{P}(\Omega)$ and  the {\bf right adjoint}
$$
R:\mathbf{Set}
\rightarrow\mathbf{Meas}
$$
endows every set $\Omega$ with the {\bf trivial} or {\bf indiscrete} $\sigma$-algebra, i.e. $\{\emptyset,\Omega\}$.
\begin{prop}
The functors $L$ and $R$, as defined above, satisfy 
$$
UL=UR=\id_{\mathbf{Set}},
$$ i.e. they are right inverses to $U$. Further, both functors $L$ and $R$ fully embed $\mathbf{Set}$ into $\mathbf{Meas}$.
\end{prop}
\begin{proof}
As every set-theoretic function between spaces endowed with the discrete or indiscrete $\sigma$-algebra is measurable, both functors give full embeddings of $\mathbf{Set}$ into $\mathbf{Meas}$.
\end{proof}

The category of all measurable structures on a fixed set $X$, i.e. all possible $\sigma$-algebras, is called the {\bf fibre of the functor $U$ above X}. It has the structure of a {\bf complete lattice} when given the order defined by inclusion. The {\bf greatest element} in the fibre $U^{-1}(X)$ is $\mathcal{P}(X)$, the powerset $\sigma$-algebra on $X$ and the {\bf least element} is the trivial / indiscrete $\sigma$-algebra $\{\emptyset,X\}$.
Therefore we have 
\begin{prop}
The category $\mathbf{Meas}$ is {\bf fibre-complete}. 
\end{prop}

For $X,Y\in\mathbf{Set}$, the equaliser of two functions (morphisms) $f,g:X\rightarrow Y$ is given by
$$
\operatorname{Eq}(f,g):=\{x\in X | f(x)=g(x)\}=:\{f=g\}
$$
The co-equaliser for $f,g\in\Hom_{\mathbf{Set}}(X,Y)$ is given by the quotient of $Y$ by the equivalence relation $\sim$ generated by the relations $f(x)\sim g(x)$ for all $x\in X$, i.e.
$$
\operatorname{co-Eq}(f,g):=Y/_{\sim}=:\{f\sim g\}
$$

Analogously to $\mathbf{Top}$, we have
\begin{thm}
The category $\mathbf{Meas}$ has the following properties:
\begin{enumerate}
\item The {\bf product} in $\mathbf{Meas}$ is given by the product $\sigma$-algebra on the set-theoretic cartesian product.
\item The {\bf co-product} is given by the disjoint union of the measurable spaces and with the disjoint union $\sigma$-algebra.
\item The {\bf final object} is $(\star,\{\emptyset, \star\})$, i.e. the one-point set with the trivial $\sigma$-algebra.
\item The {\bf equaliser} of two morphisms is given by the set-theoretic equaliser and the {\bf trace $\sigma$-algebra}.
\item The {\bf co-equaliser} is given by the set-theoretic co-equaliser and the quotient $\sigma$-algebra $\mathcal{F}_{Y/_{\sim}}:=\{A\subset Y/_{\sim}|\pi^{-1}(A)\in\mathcal{F}_Y\}$, where $\pi:Y\rightarrow Y/_{\sim}$ is the canonical projection.
\item {\bf Direct limits} and {\bf inverse limits} are the set-theoretic limits with the {\bf final} and {\bf initial $\sigma$-algebra}, respectively.  
\end{enumerate}
\end{thm}
\begin{proof}
For two measurable spaces $(\Omega_1,\mathcal{F}_1)$ and $(\Omega_2,\mathcal{F}_2)$, the measurable space  $(\Omega_1\times\Omega_2,\mathcal{F}_1\otimes\mathcal{F}_2)$ is their categorical product. Namely, let $C:=(\Omega,\mathcal{A})$ be a measurable space and $f:C\rightarrow A$ and $g:C\rightarrow B$ two morphisms, i.e. measurable maps. For $x\in\Omega$, let
\begin{equation*}
h(x):=(f(x),g(x))
\end{equation*}
be the unique set-theoretic map which is also measurable. Then for $A_i\in\mathcal{F}_i$, $i=1,2$ we have $h^{-1}(A_1\times A_2)=f^{-1}(A_1)\cap g^{-1}(A_2)\in\mathcal{A}$ as $f$ and $g$ are both $\mathcal{F}_i-\mathcal{A}$-measurable. Since, the rectangles $A_1\times A_2$ generate the product $\sigma$-algebra, the statements follows. 
\end{proof}
Finally, the ``Factorisation Lemma", cf. e.g.~[\cite{HT}~Satz~1.2~i)], is an example of a push-out in $\mathbf{Meas}$ which holds in a special but important set-up. 
\section{Itô calculus}
\subsection{Continuous semi-martingales}
We recall from~\cite{HT,IW} the general set-up and main results which have originally been introduced by K.~Itô~\cite{I}. We fix $(\Omega,\mathcal{F},(\mathcal{F})_{t\in\R_+},\P)$ a filtered probability space and consider the following vector spaces of real-valued stochastic processes defined on it:
\begin{eqnarray*}
\mathcal{B}&:=&\{X| \text{$(\mathcal{F}_t)_{t\in\R_+}$-adapted, left continuous, bounded on $[0,t]$ $\forall t\in\R_+$ (locally bounded) }\}\\
\mathcal{C}&:=&\{X\in\mathcal{B}|\text{$X$ path-wise continuous}\}\\
\mathcal{A}_{++}&:=&\{X\in\mathcal{C}  |\text{$X_0=0, X_t$ is strictly increasing and $\lim_{t\to\infty}X_t=+\infty$, a.s.}\}\\
\mathcal{A}&:=&\{X\in\mathcal{C}|\text{$X$ path-wise of bounded local variation}\}\\
\mathcal{M}&:=&\{ X\in\mathcal{C} | \text{$X-X_0\equiv(X_t-X_0)_{t\in\R_+}$ is a local martingale} \}
\end{eqnarray*} 
The space of {\bf continuous  semi-martingales} is 
$\mathcal{S}:=\mathcal{A}+\mathcal{M}$. Further, the following chain of inclusions holds: $\mathcal{A}\cup\mathcal{M}\subset\mathcal{S}\subset\mathcal{C}\subset\mathcal{B}$. 

There is a sequence of subspaces $\mathcal{B}_0\subset\mathcal{B},\dots,\mathcal{S}_0\subset\mathcal{S}$ consisting of the respective stochastic processes $X$ with $X_0=0$. 

We have $\mathcal{M}_0\cap\mathcal{A}=\{0\}$ and the fundamental {\bf Doob-Meyer} decomposition
$$
S=\mathcal{M}_0\oplus\mathcal{A}.
$$ 

For $s\leq t$ let $\mathcal{Z}_n[s,t]:s=t_0<\dots<t_n=t$ be a partition of the interval $[s,t]$ and for $[0,t]$ we shall simply write $\mathcal{Z}_n(t)$ or just $\mathcal{Z}_n$. For $X,Y\in\mathcal{C}$, the {\bf quadratic covariation} (with respect to $\mathcal{Z}_n$) is given by 
$$
[X,Y]_{\mathcal{Z}_n}:=\sum_{j=1}^n(X_{t_j}-X_{t_{j-1}})(Y_{t_j}-Y_{t_{j-1}})\qquad (\text{$:=0$ for $t=0$}).  
$$
Then $[X,Y]_{\mathcal{Z}_n}:\Omega\rightarrow\R$ and the limiting process satisfies $[X,Y]\in\mathcal{A}_0$.
\subsection{The stochastic $\hbar$-integral}
We recall the notion of a parameter dependent stochastic integral, whose construction we present in slightly generalised form, cf.~[\cite{KP}, p.~99 ff.].
 
For $t\geq0$ consider a sequence of partitions $\mathcal{Z}_n:0=t^{(n)}_0<t^{(n)}_1<\dots<t^{(n)}_n=t$, for which the {\bf mesh size} $|\mathcal{Z}_n|$ tends to zero, i.e.
$$
\lim_{n\to\infty}|\mathcal{Z}_n|=\lim_{n\to\infty}\max_{1\leq j\leq n}\left(t^{(n)}_{j}-t^{(n)}_{j-1}\right)=0.
$$

For a fixed $\hbar\in[0,1]$, $X,Y\in\mathcal{S}_0$ (just for simplicity now) and $t\in\R_+$, we define path-wise a sequence of random variables $(S_n(t))_{n\in\N^*}$, given by
\begin{equation*}
S_n(X,Y,t,\omega):=\sum_{j=1}^{n}\left((1-\hbar)X(t^{(n)}_{j-1},\omega)+\hbar X(t^{(n)}_{j},\omega)\right)\cdot\left(Y(t^{(n)}_j,\omega)-Y(t^{(n)}_{j-1},\omega)\right).
\end{equation*}
The {\bf stochastic $\hbar$-integral} is then given by
\begin{equation}
\label{lambda_int}
(\hbar)\int_0^{t}X(s,\omega)dY_s(\omega):=\lim_{n\to\infty}S_n(X,Y,t,\omega).
\end{equation}
For $\hbar=0$, the expression~(\ref{lambda_int}) is called the {\bf Itô integral}, and for the symmetric case $\hbar=1/2$, the {\bf Fisk-Stratonovich integral}. Let us use the following notation
\begin{equation}
X\bullet_{\hbar} Y:=\int_0^tX\bullet_{\hbar}dY:=(\hbar)\int_0^{t}X(s)dY_s,
\end{equation}
with the (standard) notational conventions $\bullet:=\bullet_0$ for the Itô and $\bullet_{1/2}$ for the Fisk-Stratonovich integral, respectively. In the probability literature,  $\circ$ is often used to denote Stratonovich integration.
\begin{prop}
The {\bf $\hbar$-stochastic integral} has the following representation in terms of the Itô integral $\bullet$ and the quadratic covariation $[,]$:
\begin{equation}
X\bullet_{\hbar} Y=X \bullet Y+\hbar[X,Y].
\end{equation}
\end{prop}
\begin{proof} The claim follows from the following calculation:
\begin{eqnarray*}
S_n(X,Y,t,\omega)&=&(1-\hbar)\sum_{j=1}^{n}X(t^{(n)}_{j-1},\omega)\cdot\left(Y(t^{(n)}_{j},\omega)-Y(t^{(n)}_{j-1},\omega)\right)\\
&& +\hbar\sum_{j=1}^{n} X(t^{(n)}_{j},\omega)\cdot\left(Y(t^{(n)}_{j},\omega)-Y(t^{(n)}_{j-1},\omega)\right)\\
&``="&(1-\hbar) X\bullet Y+\hbar\sum_{j=1}^{n} X(t^{(n)}_{j-1},\omega)\cdot\left(Y(t^{(n)}_{j},\omega)-Y(t^{(n)}_{j-1},\omega)\right)\\
&&+\hbar\sum_{j=1}^{n}\left(X(t^{(n)}_{j},\omega)-X(t^{(n)}_{j-1},\omega)\right)\cdot\left(Y(t^{(n)}_{j},\omega)-Y(t^{(n)}_{j-1},\omega)\right)
\end{eqnarray*}
\end{proof}
Let us recall some basic properties of the Itô calculus which we shall need most often, cf. e.g.~[\cite{HT}, Chapter~4]: the following properties hold:
\begin{eqnarray} 
\bullet&:&\mathcal{S}\otimes_{\R}\mathcal{S}\rightarrow\mathcal{S}_0\quad\text{(bilinear)}\\
{[~,~]}&:&\mathcal{S}\otimes_{\R}\mathcal{S}\rightarrow\mathcal{A}_0\quad\text{(bilinear and symmetric)}
\end{eqnarray}
and for $F,G,X,Y\in\mathcal{S}$, one has:
\begin{eqnarray}
F\bullet (G\bullet X)&=&(F\cdot G)\bullet X,\\
F\bullet[X,Y]&=&[F\bullet X,Y]=[X,F\bullet Y].
\end{eqnarray}
\begin{prop}
\label{h-associative}
For  $\hbar\in[0,1]$ and $X,Y,Z\in\mathcal{S}$ the following ``associativity" relation holds:
\begin{equation}
X\bullet_{\hbar}(Y\bullet_{\hbar} Z)=(X\cdot Y)\bullet_{\hbar} Z.
\end{equation}
\end{prop}
\begin{proof}
For $X,Y,Z\in\mathcal{S}$, by using integration by parts, we have 
\begin{eqnarray*}
X\bullet Y+Y\bullet X&=&X\cdot Y-X_0Y_0-[X,Y]\qquad |[\cdot,Z]\\
{[X\bullet Y, Z]}+{[Y\bullet X,Z]} &=& [X\cdot Y-X_0Y_0,Z]-[[X,Y],Z]\\
&=&[X\cdot Y,Z]
\end{eqnarray*}
With this we have
\begin{eqnarray*}
X\bullet_{\hbar}(Y\bullet_{\hbar} Z)&=&X\bullet(Y\bullet Z+\hbar[Y,Z])+\hbar[X,Y\bullet Z+\hbar[Y,Z]]\\
&=&(X\cdot Y)\bullet Z+\hbar X\bullet[Y,Z]+\hbar[X,Y\bullet Z]+\hbar^2[X,[Y,Z]]\\
&=&(X\cdot Y)\bullet Z+\hbar([X\bullet Y,Z]+[Y\bullet X,Z])\\
&=&(X\cdot Y)\bullet Z+\hbar[X\cdot Y,Z]\\
&=&(X\cdot Y)\bullet_{\hbar} Z
\end{eqnarray*}
\end{proof}

K. Itô showed~\cite{I}, cf. also~[\cite{IW}, Chapter III], that the space of semi-martingales carries remarkable algebraic structures. 

\begin{thm}[K.Itô~\cite{I}]
\label{Ito_Thm}
The set $\mathcal{S}_I:=(\mathcal{S},\cdot)$ is an associative, commutative $\R$-algebra with respect to point-wise multiplication. The unit is given by the constant process $1_{\mathcal{S}_I}$. The set $(\mathcal{S},[,])$ is an associative and commutative $\mathcal{S}_I$-algebra which is nilpotent of order two.
\end{thm}

Further, he established the module structures over $\mathcal{S}_I$ for both the Itô and the Stratonovich stochastic integral. We note that 
the subspaces $\mathcal{M}_0$ and $\mathcal{A}$ are $\mathcal{S}_I$-submodules. A natural generalisation, which contains both cases, is given in 
\begin{prop}
$(\mathcal{S}_0,\bullet_{\hbar})$ is a left module over the $\R$-algebra $\mathcal{S}_I$.
\end{prop}
\begin{proof}
$X\bullet_{\hbar}(Y\bullet_{\hbar} Z)=(X\cdot Y)\bullet_{\hbar} Z$ and $1_{\mathcal{S}_I}\bullet_{\hbar} Z=1_{\mathcal{S}_I}\bullet Z+\hbar[1_{\mathcal{S}_I},Z]=Z.$
\end{proof}

\section{Operadic structures and cohomology}
\subsection{Dendriform and Tridendriform algebras}
\label{operads}
Here we shall restrict our discussion to the subspace $\mathcal{S}_0$, if not otherwise stated. 

Let us define the following three $\R$-bilinear operations $\prec_I,\succ_I,\cdot_I:\mathcal{S}_0\otimes_{\R}\mathcal{S}_0\rightarrow\mathcal{S}_0$ by:

\begin{eqnarray}
\label{Ito-tridend1}
X\prec_I Y&:=&Y\bullet X\\
\label{Ito-tridend2}
X\succ_I Y&:=&X\bullet Y\\
\label{Ito-tridend3}
X\cdot_I Y&:=& [X,Y]
\end{eqnarray}
and let 
\begin{equation}
\label{prod1}
X\ast Y:=X\prec_I Y+X\succ_I Y+X\cdot_I Y.
\end{equation}
Further, for every $\hbar\in[0,1]$, we define two $\R$-bilinear operations $\dashv^{!}_{\hbar}, \vdash^{!}_{\hbar}:\mathcal{S}_0\otimes_{\R}\mathcal{S}_0\rightarrow\mathcal{S}_0$, by:
\begin{eqnarray}
\label{Ito-dend1}
X\dashv^{!}_{\hbar} Y&:=&X\prec_I Y+(1-\hbar)(X\cdot_I Y)=Y\bullet X+(1-\hbar)[X,Y]\\
\label{Ito-dend2}
X\vdash^{!}_{\hbar} Y&:=& X\succ_I Y+\hbar (X\cdot_I Y)\;\;\qquad=X\bullet Y+\hbar[X,Y]
\end{eqnarray}
and let
\begin{equation*}
X\star Y:= X\dashv^{!}_{\hbar} Y+X\vdash_{\hbar}^{!} Y.
\end{equation*}
We use the {\bf shriek $!$} for reasons of Koszul duality which should become clear below.

Integration by parts gives 
\begin{equation}
\label{Intpart_equal}
X\ast Y=X\cdot Y=X\star Y.
\end{equation}

\begin{thm}
The Itô calculus defines on $\mathcal{S}_0$, with the definitions given in~(\ref{Ito-tridend1}),(\ref{Ito-tridend2},(\ref{Ito-tridend3})) and (\ref{prod1}), the structure of a {\bf commutative tridendriform algebra}, i.e. we have (in mixed notation):
\begin{eqnarray*}
(X\prec_I Y)\prec_I Z&=&X\prec_I(Y\ast Z)\\
(X\succ_I Y)\prec_I Z&=&X\succ_I (Y\prec_I Z)\\
(X\ast Y)\succ_I Z&=&X\succ_I (Y\succ_I Z)\\
{[X\succ_I Y,Z]}&=&X\succ_I[Y,Z]\\
{[X\prec_I Y,Z]}&=&{[X,Y\succ_I Z]}\\
{[X,Y]\prec_I Z} &=&{[X,Y\prec_I Z]}\\
{[[X,Y],Z]}&=&{[X,[Y,Z]]}
\end{eqnarray*}
and commutativity: 
\begin{equation*}
X\prec_I Y=Y\succ_I X \qquad\text{and}\quad [X,Y]=[Y,X]
\end{equation*}
\end{thm}
\begin{proof}
The commutativity follows from the above definitions and the commutativity of the quadratic covariation. 
The other relations follow again from the definitions and the properties of the stochastic integral, cf.~[\cite{HT}, Satz 4.38.]
\end{proof}

Loday~\cite{L07} showed that commutative tridendriform algebras are intimately connected with {\bf quasi-shuffle algebras} which have the algebraic structure of iterated Itô integrals, as has been previously shown  by J.~Gaines~\cite{G}.

\begin{thm}
For every $\hbar\in[0,1]$, the $\hbar$-integral induces the structure of a {\bf dendriform} algebra on $\mathcal{S}_0$, i.e. 
\begin{eqnarray*}
(X\dashv^{!}_{\hbar} Y)\dashv^!_{\hbar} Z&=& X\dashv^!_{\hbar}(Y\dashv^!_{\hbar} Z)+X\dashv^{!}_{\hbar}(Y\vdash^{!}_{\hbar} Z),\\
(X\vdash^{!}_{\hbar} Y)\dashv^{!}_{\hbar} Z&=& X\vdash^{!}_{\hbar} (Y\dashv^{!}_{\hbar} Z),\\
(X\dashv^{!}_{\hbar} Y)\vdash^{!}_{\hbar} Z+(X\vdash^{!}_{\hbar} Y)\vdash^{!}_{\hbar} Z &=& X\vdash^{!}_{\hbar} (Y \vdash^{!}_{\hbar} Z).
\end{eqnarray*}
For $\hbar=1/2$, i.e. the Stratonovich integral, it is commutative and therefore also a {\bf Zinbiel} algebra.
\end{thm}
\begin{rem}
For semimartingales with jumps, some but not all of the above relations hold. Also, one might consider parametrised versions of the dendriform and tridendriform structures, cf.~\cite{LV}. Further, one can naturally extend the present framework to algebras up to homotopy.
\end{rem}
\begin{proof}
Follows again by direct verification by using the above definitions and the properties of the stochastic integral.
\end{proof}
\begin{rem}
The Riemann-Stieltjes integral defines the structure of a commutative dendriform algebra, and hence a Zinbiel algebra.
\end{rem}

As every dendriform algebra defines a left and a right pre-Lie structure, we have 
\begin{prop}
The space $\mathcal{S}_0$ has a left and right pre-Lie product, given by: 
\begin{eqnarray*}
L\{X, Y\}_{\hbar}:=X\vdash^{!}_{\hbar} Y-Y\dashv^{!}_{\hbar} X=(2\hbar-1)[X,Y],\\
R\{X, Y\}_{\hbar}:=X\dashv^{!}_{\hbar}Y-Y\vdash^{!}_{\hbar} X=(1-2\hbar)[X,Y].
\end{eqnarray*}
For the Stratonovich integral both pre-Lie structures vanish.
\end{prop}

\subsection{Deformations}
Deformations of modules are, according to the work of M. Gerstenhaber, characterised by co-cycles in the Hochschild cohomology. Here we follow the approach in~\cite{Y}, cf. also the review~\cite{Du}.

Let $(R,\alpha)$ be a $k$-algebra, with $k$ a field of characteristic zero, and $(M,\xi)$ a left $R$-module, i.e. $\xi$ is a map of associative $k$-algebras $\xi:R\rightarrow\End_k(M)$, where $(\End_k(M),\circ)$ is endowed with the associative $k$-algebra structure given by composition of $k$-linear endomorphisms. 

A {\bf formal deformation} of $\xi$ is given by a power series
\begin{equation*}
\xi_{\hbar}=\xi+\hbar\xi_1+\hbar^2\xi_2+\cdots,
\end{equation*}
where for every $i\in\N^*$, $\xi_i\in\Hom_k(R,\End_k(M))$, i.e. $\xi_i$ is a $k$-linear map from the $k$-algebra $R$ to the $k$-module $\End_k(M)$, such that the following multiplicative condition is satisfied:
\begin{equation*}
\xi_{\hbar}(rs)=\xi_{\hbar}(r)\circ\xi_{\hbar}(s)\qquad\forall r,s\in R,
\end{equation*}
i.e. $\xi_{\hbar}\in \Hom_{\mathbf{Alg}_k}(R,\End_k(M))$. Then one calls the $k$-linear map $\xi_1$ the {\bf infinitesimal deformation} of $\xi_{\hbar}$. We note that by definition $\xi$ is multiplicative, but not necessarily the $\xi_i$'s. 
Now, for $n\in\N$, we have
\begin{equation*}
C^n(\mathcal{S}_I,\mathcal{S}_0)=\Hom_{\R}(\mathcal{S}_I^{\otimes n}\otimes \mathcal{S}_0,\mathcal{S}_0)
\end{equation*}
with the corresponding $\N$-graded $\R$-module
$$
C^{\bullet}(\mathcal{S}_I,\mathcal{S}_0):=\bigoplus_{n=0}^{\infty}C^n(\mathcal{S}_I,\mathcal{S}_0),
$$
i.e. the {\bf Hochschild complex}. The differential $d$ has components \begin{equation*}
d_{n}(f):=\xi(\mathbbm{1}_{\mathcal{S}_I}\otimes f)+\sum_{i=1}^n(-1)^{i} f(\mathbbm{1}^{\otimes i-1}_{\mathcal{S}_I}\otimes\alpha\otimes \mathbbm{1}^{\otimes n-i}\otimes \mathbbm{1}_{\mathcal{S}_0})+(-1)^{n+1}f(\mathbbm{1}^{\otimes n}_{\mathcal{S}_I}\otimes \xi),
\end{equation*}
for all $n\in\N$ and $f\in C^n(\mathcal{S}_I,\mathcal{S}_0)$.

Let 
\begin{equation*}
\xi:\mathcal{S}_I\rightarrow\End_{\R}(\mathcal{S}_0),\quad X\mapsto X\bullet-,
\end{equation*}
which by~[\cite{HT}, Satz 4.38] is multiplicative, and 
\begin{equation*}
\xi_1:\mathcal{S}_I\rightarrow\End_{\R}(\mathcal{S}_0), \quad X\mapsto[X,-].
\end{equation*} 
which is just $\R$-linear. 
\begin{thm}
The $\hbar$-integral $\bullet_{\hbar}:=\bullet+\hbar[~,]$ is a first order (infinitesimal) deformation of the Itô integral $\bullet$. The quadratic covariation $[,]$ gives a {Hochschild \bf $1$-cocycle}, i.e. an element of $Z^1(\mathcal{S}_I,\mathcal{S}_0)$. 
\end{thm} 
\begin{proof}
By Proposition~\ref{h-associative}, $\xi_{\hbar}:=\xi+\hbar\xi_1$ is multiplicative, which implies that $\xi_1$ is infinitesimal. Then by~[\cite{Y}, Proposition~3.1] the claim follows.
\end{proof}

It seems that the deformation problem is {\bf not rigid}, i.e. there exists no $g\in\End_{\R}(\mathcal{S}_0)$ such that for $F\in\mathcal{S}_I$ and $X\in\mathcal{S}_0$, the equation 
\begin{equation*}
\label{not-rigid}
[F,X]=F\bullet g(X)-g(F\bullet X)
\end{equation*}
holds, i.e. $H^1(\mathcal{S}_I,\mathcal{S}_0)\neq 0$.

\section{Operad morphisms and Lie groupoid structures}
\subsection{The stopping time operator}
Stopping times are comparable to contour lines, i.e. like annual rings of a tree. We shall use  the common notations $\wedge:=\min$ (meet) (and $\vee:=\max$ (join)).
Let $X$ be stochastic process and $\tau$ a stopping time. The {\bf stopped process} $\tau(X):=(X_{t\wedge\tau})_{t\in\R_+}$ is given by
$$
\tau(X_t)(\omega):=X_{t\wedge\tau}(\omega):=
\begin{cases}
      & X_{t}(\omega)\quad\text{for $t\leq\tau(\omega)$}, \\
      & X_{\tau(\omega)}(\omega)\quad\text{for $t>\tau(\omega)$}.
\end{cases}
$$
i.e. after time $\tau$ the paths of $\tau(X)$ are constant.

It is known~\cite{Fr} that the set of stopping times forms a lattice when partially ordered with: $\tau<\tau'$ if $\{\tau>t\}\subseteq\{\tau'>t\}$. 

Every stopping time $\tau$ defines an $\R$-linear operator 
$\tau:\mathcal{C}\rightarrow\mathcal{C}$,
$\tau(\alpha X+\beta Y)=\alpha\tau(X)+\beta\tau(Y)$, for $\alpha,\beta\in\R$
 which further respects the Doob-Meyer decomposition, i.e. $\tau(\mathcal{M}_0)\subset\mathcal{M}_0$ and $\tau(\mathcal{A})\subset\mathcal{A}$. 
 \begin{prop}
For every $\hbar\in[0,1]$, each stopping time is an endomorphism of the dendriform and tridendriform structure on $\mathcal{S}_0$, i.e. we have:
\begin{eqnarray*}
\tau(X\dashv^{!}_{\hbar} Y)&=&\tau(X)\dashv^{!}_{\hbar}\tau(Y)\\
\tau(X\vdash^{!}_{\hbar} Y)&=&\tau(X)\vdash^{!}_{\hbar}\tau(Y)
\end{eqnarray*}
and
\begin{eqnarray*}
\tau(X\prec_I Y)&=&\tau(X)\prec_I\tau(Y)\\
\tau(X\succ_I Y)&=&\tau(X)\succ_I\tau(Y)\\
\tau(X\cdot_I Y)&=&\tau(X)\cdot_I \tau(Y)
\end{eqnarray*}
\end{prop}
\begin{rem}
For $F\in\mathcal{S}_I, X\in\mathcal{S}_0$ we have the $\mathcal{S}_I$-module morphism, 
$\tau(F\bullet X)=F\bullet\tau(X )$. Further,  the following relations are also valid,  $\tau[X,Y]=[X,\tau(Y)]=[\tau(X),Y]$, cf.~[\cite{HT}, Satz 4.35.].
\end{rem}
\begin{proof}
This follows by direct calculations from the properties given in~[\cite{HT}, Korollar p.~173 and Satz~4.35.] and the definitions in Section~\ref{operads}.
\end{proof}
\subsection{The time change operator}
Here we do not consider the most general notion of time change but a particularly adapted one. References for this section are~[\cite{IW}, Chapter~III], [\cite{HT}, Chapter~5.2] and~\cite{CdSW,W2}. We fix a standard filtered probability space $(\Omega,\mathcal{F},(\mathcal{F}_t)_{t\in\R_+},\P)$.

The set $\mathcal{A}_{++}$ is a {\bf blunt cone}, i.e. $\forall \lambda>0$ and $\phi\in\mathcal{A}_{++}$ we have $\lambda \phi\in\mathcal{A}_{++}$. We note that these elements have zero variation and are called {\bf time changes}.

Every time change process $\phi$ defines an $(\mathcal{F}_t)$-stopping time process $(\tau_{\phi})_t$, or just $\tau_t$, by
\begin{equation*}
(\tau_{\phi})_t:=\inf\{u\in\R_+ |\phi_u>t\}.
\end{equation*}
Then $(\tau_{\phi})_0=0$, the map $t\mapsto(\tau_{\phi})_t$ is continuous and strictly increasing and $\lim_{t\to\infty}(\tau_{\phi})_t=+\infty$, a.s. 

The time-transformed filtration
\begin{equation*}
\mathcal{F}_{{\phi}_t}:=\mathcal{F}_{(\tau_{\phi})_t}\qquad\text{for $t\in\R_+$},
\end{equation*}
induces a new filtration on $(\Omega,\mathcal{F},\P)$ and hence determines the corresponding vector spaces $\mathcal{C}_{\phi}, \mathcal{A}_{\phi},\mathcal{M}_{\phi}$ and $\mathcal{S}_{\phi}$. 

Every $\phi\in\mathcal{A}_{++}$ defines an $\R$-linear operator 
$$
T_{\phi}:\mathcal{S}\rightarrow\mathcal{S}_{\phi},\qquad (T_{\phi}(X))_t:=X_{(\tau_{\phi})_t}\quad\forall X\in\mathcal{S},
$$ 
or in more detail, $X_{(\tau_{\phi})_t}(\omega,t)=X_{\tau_{(\phi(\omega))_t}}(\omega)$.
\begin{prop}[\cite{IW}, p.102]
\label{time_change_IW}
$T_{\phi}:\mathcal{S}\rightarrow\mathcal{S}_{\phi}$ is a linear isomorphism which respects the Doob-Meyer decomposition, as shown in the commutative diagram below:
\[
\begin{xy}
  \xymatrix{
     \mathcal{S}\ar@<-3pt>[d]_{\pr_1}\ar@<+2pt>[d]^{\pr_2}\ar[rr]^{T_{\phi}} &   &\mathcal{S}_{\phi}\ar@<-3pt>[d]_{\pr_1}\ar@<+2pt>[d]^{\pr_2} &&  \\
                            \mathcal{M}_0\oplus\mathcal{A}\ar[rr]^{T_{\phi}}   &  & (\mathcal{M}_{\phi})_0\oplus\mathcal{A}_{\phi}               }
\end{xy}
\]
i.e. for $X\in\mathcal{S}$ we have $\pr_i(T_{\phi}(X))=T_{\phi}|_i(\pr_i(X))$, $i=1,2$.  Further, $T_{\phi}$ is an isomorphism of the $\mathcal{S}_I$-algebra / module $\mathcal{S}_0$ with the $(\mathcal{S}_{\phi})_I$-algebra / module $(\mathcal{S}_{\phi})_0$.
\end{prop}

\begin{prop}
For every $\phi\in\mathcal{A}_{++}$, the associated time change operator $T_{\phi}$ defines an isomorphism between the dendriform and tridendriform operad structures on $\mathcal{S}_0$ and $(\mathcal{S}_{\phi})_0$, i.e.
\begin{eqnarray*}
T_{\phi}(X\dashv^{!}_{\hbar} Y)&=&T_{\phi}(X)\dashv^{!}_{\hbar} T_{\phi}(Y)\\
T_{\phi}(X\vdash^{!}_{\hbar} Y)&=&T_{\phi}(X)\vdash^{!}_{\hbar} T_{\phi}(Y)
\end{eqnarray*}
and
\begin{eqnarray*}
T_{\phi}(X\prec_I Y)&=&T_{\phi}(X)\prec_I T_{\phi}(Y)\\
T_{\phi}(X\succ_I Y)&=&T_{\phi}(X)\succ_I T_{\phi}(Y)\\
T_{\phi}(X\cdot_I Y)&=&T_{\phi}(X)\cdot_I T_{\phi}(Y)
\end{eqnarray*}
\end{prop}
\begin{proof}
This follows by direct verification from the definitions in Section~\ref{operads} and Proposition~\ref{time_change_IW}.
\end{proof}
Define the stochastic process $\id_t\in\mathcal{A}_{++}$ as $\id_t(\omega):=t$ for all 
$\omega\in\Omega$ and let $\mathfrak{S}_{\mathcal{A}_{++}}:=\bigcup_{\phi\in\mathcal{A}_{++}}\mathcal{S}_{\phi}$ be the disjoint union of all semimartingales obtained by a time-change. Then the fibre over $\id_{t}$ is  $\mathcal{S}$ and $\mathfrak{S}_{\mathcal{A}_{++}}$ is a real vector bundle over $\mathcal{A}_{++}$.

We define the following small category $G_{\mathcal{A}_{++}}$: 

the set of objects is $\mathfrak{S}_{\mathcal{A}_{++}}$, and for $\mathcal{S}_{\phi},\mathcal{S}_{\psi}\in\mathfrak{S}_{\mathcal{A}_{++}}$, we let the set of arrows be $G(\mathcal{S}_{\phi},\mathcal{S}_{\psi}):=\{T_{\phi,\psi}\}$ where
\begin{equation}
\label{time_change_mor}
T_{\phi,\psi}:=T_{\psi}\circ T^{-1}_{\phi}:\mathcal{S}_{\phi}\rightarrow\mathcal{S}_{\psi},
\end{equation}
and hence $G(\mathcal{S}_{\phi},\mathcal{S}_{\phi})=\{\id_{\mathcal{S}_{\phi}}\}$. Then $\mathcal{S}_{\phi}$ is called the {\bf source} and $\mathcal{S}_{\psi}$ the {\bf target}.

Further, for every pair $\mathcal{S}_{\phi},\mathcal{S}_{\psi}\in\mathfrak{S}_{\mathcal{A}_{++}}$, we have a map $\operatorname{inv}:G(\mathcal{S}_{\phi},\mathcal{S}_{\psi})\rightarrow G(\mathcal{S}_{\psi},\mathcal{S}_{\phi})$ given by
\begin{equation*}
T_{\phi,\psi}\mapsto T_{\psi,\phi}=T^{-1}_{\phi,\psi}.
\end{equation*}

\begin{prop}
$G_{\mathcal{A}_{++}}$ is a groupoid, i.e. a time change, as defined above, induces a groupoid structure on $\mathfrak{S}_{\mathcal{A}_{++}}$.
\end{prop}
\begin{rem}
In the appropriate framework we are in fact dealing with a Lie groupoid. 
\end{rem} 
\begin{proof}
We have to show that the composition of triples is well-defined. For three objects $\mathcal{S}_{\phi_i}\in\mathfrak{S}_{\mathcal{A}_{++}}$, $i\in\{1,2,3\}$, the function 
\begin{equation*}
\operatorname{comp}_{\phi_1,\phi_2,\phi_3}:G(\mathcal{S}_{\phi_2},\mathcal{S}_{\phi_3})\times G(\mathcal{S}_{\phi_1},\mathcal{S}_{\phi_2})\rightarrow G(\mathcal{S}_{\phi_1},\mathcal{S}_{\phi_3}),\quad (T_{\phi_2,\phi_3},T_{\phi_1,\phi_2})\mapsto T_{\phi_2,\phi_3}\circ T_{\phi_1,\phi_2}
\end{equation*}
is well-defined with~(\ref{time_change_mor}), also the associativity of the composition and  the unit property hold, as we are composing linear operators. 
\end{proof}
\subsection{The Girsanov operator}
For some of the measure theoretic details one should consult~[\cite{HT}, Section~5.2], which is also the main reference for this section. Additionally, we use~\cite{CdSW,D,W2}.

Let $(\Omega,\mathcal{F},\P,(\mathcal{F}_t)_{t\in\R_+})$ be an arbitrary filtered probability space. 
The ideal $\mathcal{N}\subset\mathcal{F}$ of $\P$-null sets is idempotent, i.e. $\mathcal{N}^2=\mathcal{N}$, and it gives rise to a short exact sequence
\[
\begin{xy}
  \xymatrix{
     0\ar[r]&\mathcal{N}\ar[r]^{\iota}&\mathcal{F} \ar[r]^p&\mathcal{F} /\mathcal{N}\ar[r]&0}
\end{xy}
\]
Let us recall the difference between an ``augmentation" and a ``completion". For a sub-$\sigma$-algebra $\mathcal{A}\subset\mathcal{F}$, the {\bf augmentation} of $\mathcal{A}$ by $\mathcal{N}$ is defined as the $\sigma$-algebra $\mathcal{A}^*:=\sigma(\mathcal{A}\cup\mathcal{N})$ which is equal to $\mathcal{A}+\mathcal{N}$ where $``+"$ is given by the symmetric difference of sets. The {\bf completion} of $\mathcal{F}$ is formed by adjoining all subsets of $\P$-null sets to $\mathcal{F}$ and then extending the measure to the resulting $\sigma$-algebra.

Finally, a filtration $(\mathcal{F}_t)_{t\in\R_+}$ is called {\bf complete} if $\mathcal{F}_0=\mathcal{F}^*_0$, i.e.  $\mathcal{N}\subset\mathcal{F}_0$, and hence also $\mathcal{F}_t=\mathcal{F}^*_t$.

Here we need the concept of a {\bf local completion}, cf.~[\cite{HT}, p.~250], which consists of adjoining certain families of subsets to the original $\sigma$-algebra and then extending the measure.

We recall from~[\cite{HT}, Section~5.2] the general framework:  let $(\Omega,\mathcal{F},\P,(\mathcal{F})_{t\in\R_+})$ be a filtered probability space and $D\geq0$, a martingale  such that $\int_{\Omega} D_0(\omega) d\P(\omega)=1$. This defines a {\bf projective family} of probability measures $Q_*$, with $Q_t:=D_t\cdot\P$ on $\mathcal{F}_t$ and $Q_s=Q_t|\mathcal{F}_s$ for $s\leq t$. Assuming $\sigma$-additivity of $Q_*$ and local completeness of the filtration, we shall deal with the following situation only: we consider mutually absolutely continuous probability measures 
\begin{equation*}
Q\sim_{\operatorname{loc}}\P\quad\text{with martingale density $D_{Q\P}>0$},
\end{equation*}
such that $Q_t\sim\P_t$ for all $t\in\R_+$. 

Let $\operatorname{Prob}(\mathcal{F})$ denote the convex set of probability measures on $\mathcal{F}$, and let 
\begin{equation*}
[\P]_{\operatorname{loc}}:=[\P,(\mathcal{F}_t)_{t\in\R_+}]_{\operatorname{loc}}:=\{Q\in\operatorname{Prob}(\mathcal{F})~|~\text{$Q\sim_{\operatorname{loc}}\P$ with martingale density $D_{Q\P}>0$}\}
\end{equation*}
For $Q\in[\P,(\mathcal{F}_t)_{t\in\R_+}]_{\operatorname{loc}}$, define the {\bf Girsanov operator} $\operatorname{G}$, associated to the change of measure $\P\rightarrow Q$ with density $D=D_{Q\P}$, as:
\begin{equation*}
 \operatorname{G}_{Q,\P}:\mathcal{S}\rightarrow\mathcal{S},\quad X\mapsto X-\frac{1}{D}\bullet[X,D]~,
\end{equation*}
which is an $\R$-linear isomorphism and satisfies $\operatorname{G}|_{\mathcal{A}}=\id_{\mathcal{A}}$. 

\begin{prop}
For every $Q\in[\P,(\mathcal{F}_t)_{t\in\R_+}]_{\operatorname{loc}}$, the Girsanov operator defines an automorphism of the $\mathcal{S}_I$-module $\mathcal{S}_0$, i.e. for $F\in\mathcal{S}_I$, $X\in\mathcal{S}_0$, we have 
\begin{equation*}
\operatorname{G}_{Q,\P}(F\bullet X)=F\bullet\operatorname{G}_{Q,\P}(X).
\end{equation*}
\end{prop}
\begin{proof}
This follows from~[\cite{HT}, p.~256 Korollar] and Theorem~\ref{Ito_Thm}. 
\end{proof}

For $Q\in[\P,(\mathcal{F}_t)_{t\in\R_+}]_{\operatorname{loc}}$, let $\mathcal{S}_Q$ denote the corresponding $\R$-vector space of continuous semimartingales.

Let
\begin{equation*}
\mathfrak{S}_{[\P]_{\operatorname{loc}}}:=\mathfrak{S}_{[\P,(\mathcal{F}_t)_{t\in\R_+}]_{\operatorname{loc}}}:=\{\mathcal{S}_Q~|~Q\in[\P,(\mathcal{F}_t)_{t\in\R_+}]_{\operatorname{loc}}\}
\end{equation*}
be the disjoint union of the $\mathcal{S}_{Q}$. This set has the structure of an infinite dimensional fibre bundle with base $[\P]_{\operatorname{loc}}$. The fibre-wise Doob-Meyer decomposition defines two transversal sub-bundles $\mathfrak{M}_0$ and $\mathfrak{A}$, 
and the Girsanov operator respects this splitting, i.e. $\operatorname{G}(\mathcal{M}_{\P})=\mathcal{M}_Q$ and $G(\mathcal{A}_{\P})=\mathcal{A}_Q (=\mathcal{A}_{\P})$.

Therefore, the Girsanov operator  can be perceived either as an $\R$-linear automorphism of $\mathcal{S}$ or as a linear isomorphism of two, a priori different,  spaces $\mathcal{S}$ and $\mathcal{S}_{Q}$. 

Let us define two small categories.
The category $G_{[\P]_{\operatorname{loc}}}$ has as set of objects $[\P]_{\operatorname{loc}}$,  and as morphisms $G_{[\P]_{\operatorname{loc}}}(Q_1,Q_2):=\{D_{Q_1,Q_2}:=D_{Q_1,\P}\cdot D^{-1}_{Q_2,\P}\}$, for $Q_1,Q_2\in [\P]_{\operatorname{loc}}$ and hence $G(Q,Q)=\{1\}$. 

Further, for pairs of objects $Q_1,Q_2$, we have a map $\operatorname{inv}: G_{Q_1,Q_2}\rightarrow G_{Q_2,Q_1}$, $D_{Q_1,Q_2}\mapsto D_{Q_2,Q_1}$, and for triples of objects $Q_i$, $i\in\{1,2,3\}$, a function
$\operatorname{comp}_{Q_1,Q_2,Q_3}:G_{[\P]_{\operatorname{loc}}}(Q_2,Q_3)\times G_{[\P]_{\operatorname{loc}}}(Q_1,Q_2)\rightarrow G_{[\P]_{\operatorname{loc}}}(Q_1,Q_3)$, $(D_{Q_2,Q_3}, D_{Q_1,Q_2})\mapsto D_{Q_1,Q_3}$.

The category $G_{\mathfrak{S}_{[\P]_{\operatorname{loc}}}}$ has as set of objects $\mathfrak{S}_{[\P]_{\operatorname{loc}}}$, and as morphisms $G_{\mathfrak{S}_{[\P]_{\operatorname{loc}}}}(\mathcal{S}_{Q_1},\mathcal{S}_{Q_2}):=\{\operatorname{G}_{Q_1,Q_2}\}$, i.e. the respective Girsanov operator, and so $G_{\mathfrak{S}_{[\P]_{\operatorname{loc}}}}(\mathcal{S}_{Q},\mathcal{S}_{Q})=\{\id_{\mathcal{S}_Q}\}$.

The function $\operatorname{inv}:G(\mathcal{S}_Q,\mathcal{S}_P)\rightarrow G(\mathcal{S}_P,\mathcal{S}_Q)$ is given by $\operatorname{G}_{Q,P}\mapsto\operatorname{G}_{Q,P}^{-1}$, and a composition function 

$\operatorname{comp}_{\mathcal{S}_{Q_1},\mathcal{S}_{Q_2},\mathcal{S}_{Q_3}}:G_{[\P]_{\operatorname{loc}}}(\mathcal{S}_{Q_2},\mathcal{S}_{Q_3})\times G_{[\P]_{\operatorname{loc}}}(\mathcal{S}_{Q_1},\mathcal{S}_{Q_2})\rightarrow G_{[\P]_{\operatorname{loc}}}(\mathcal{S}_{Q_1},\mathcal{S}_{Q_3})$, $(\operatorname{G}_{Q_2,Q_3}, \operatorname{G}_{Q_1,Q_2})\mapsto \operatorname{G}_{Q_1,Q_3}$. 

\begin{thm} 
$G_{[\P,(\mathcal{F}_t)_{t\in\R_+}]_{\operatorname{loc}}}$  and $G_{\mathfrak{S}_{[\P,(\mathcal{F}_t)_{t\in\R_+}]_{\operatorname{loc}}}}$ are isomorphic infinite-dimensional Lie groupoids.
\end{thm}

\subsection*{Acknowledgments} 
The author thanks Owen Gwilliam for numerous inspiring and helpful discussions which also influenced this article. He is grateful to Gabriel Drummond-Cole for bringing Zinbiel algebras to his attention and suggesting to look at Loday's work. Further, he thanks the MPI in Bonn for its hospitality. Finally, he thanks Roland Speicher for his support and interest.

\end{document}